\documentclass[12pt,paper]{article}
\usepackage{amsmath,amssymb,amsthm,amsfonts,fullpage}
\usepackage{enumitem}  
\usepackage[
backend=bibtex,
style=numeric,
sorting=nyt
]{biblatex}
\usepackage{authblk}
\usepackage[hidelinks,pdfauthor={Abdelghaffar Chibloun}]{hyperref}
\hypersetup{
	colorlinks   = true,
	citecolor    = red
}
\usepackage{relsize}
\usepackage{cleveref}
\usepackage{color,soul}
\usepackage{color,xcolor,soul}

\soulregister\cite7
\soulregister\ref7

\providecommand{\keywords}[1]
{
	{\small	
		\textbf{{Keywords--}} #1}
}
\providecommand{\msc}[1]
{
	{\small	
		\textbf{{Mathematics Subject Classification (2020) --}} #1}
}

\theoremstyle{plain}
\newtheorem{definition}{Definition}
\newtheorem{theorem}{Theorem}
\newtheorem{proposition}{Proposition}
\newtheorem{corollary}{Corollary}
\newtheorem{lemma}{Lemma}
\newtheorem{remark}{Remark}

\newtheorem{notation}{Notation}

\DeclareRobustCommand{\rchi}{{\mathpalette\irchi\relax}}
\newcommand{\irchi}[2]{\raisebox{\depth}{$#1\chi$}} 

\usepackage{tikz-cd}

\usepackage{nicematrix}

\makeatletter
\def\@fnsymbol#1{\ensuremath{\ifcase#1\or * \or \ddagger \or \dagger \or \mathsection \or \mathparagraph \or * \else\@ctrerr\fi}}
\makeatother

\title{{
On module endomorphisms annihilated by residually square-free polynomials over Artinian local PIRs

}}

\author[1]{Abdelghaffar Chibloun\thanks{E-mail: \rm{ab.chibloun@edu.umi.ac.ma}}}
\author[2]{Mustapha Najmeddine\thanks{E-mail: \rm{m.najmeddine@umi.ac.ma}}}

\affil[1]{Department of Mathematics, FS–Meknes, Moulay Ismail University, Meknes, Morocco}
\affil[2]{Department of Mathematics, ENSAM–Meknes, Moulay Ismail University, Meknes, Morocco}
\date{}
\begin{document}

	\maketitle
	\begin{abstract}
		We consider finitely generated free modules and their endomorphisms over an Artinian local PIR, also known as an SPIR.
		We characterize submodules of such modules by constructing minimal generating sets from given bases.	We prove that any endomorphism possessing a square-free annihilating polynomial modulo the maximal ideal admits a principal annihilating ideal, and establish that every free invariant submodule under such an endomorphism has an invariant complement. Moreover,  for such endomorphisms, we detail the construction of invariant submodules and compute the number of free invariant submodules when the ring is finite. Furthermore, we introduce canonical matrices associated with these endomorphisms, yielding a complete characterization of matrix similarity via invariant polynomials in this context. Finally, we address the non-derogatory case, obtaining a sharper characterization of submodule invariance. 
	\end{abstract}
	
	
	\msc{13E10, 13F10, 13C05, 15B33, 15A21  }
	
	\section{Introduction}
	
	A \textit{special principal ideal ring} (SPIR) is a commutative principal ideal ring (PIR) having exactly one prime ideal. It is obviously local and Noetherian. Moreover, as it has Krull dimension 0, it is an Artinian local PIR. Conversely, every Artinian local PIR is an SPIR. Indeed, as such a ring has Krull dimension 0, each of its prime ideals is maximal, hence as it is local, it has only one prime ideal.	
	SPIRs play a significant role in commutative algebra, mainly due to the Zariski–Samuel Structure Theorem for PIRs  \cite[Volume 1, Chapter 4, Theorem 33]{zariski1960commutative}, which states that every PIR is isomorphic to a finite direct product of PIDs and SPIRs. 
	They arise naturally in several classical contexts. 	
	Prominent examples include \textit{finite chain rings}---such as \textit{Galois rings} $\text{GR}(p^n, r)$ (which reduce to $\mathbb{Z}/p^n\mathbb{Z}$ when $r=1$) and their \textit{Eisenstein extensions} (see \cite[Theorem XVII.5 ]{mcdonald1974finite})---as well as quotients of the form $k[x]/(f^n)$, where $k$ is a field and $f \in k[x]$ is irreducible.
	
	Study of matrices with coefficients in an SPIR has attracted considerable attention over several decades. Davis~\cite{davis1968certain} showed that over $\mathbb{Z}/p^n\mathbb{Z}$, the similarity of two square matrices reduces to the similarity of their reductions modulo $p,$ provided both matrices satisfy a common polynomial whose reduction modulo $p$ is square-free. McDonald~\cite{mcdonald1978similarity} developed the general theory of similarity over SPIRs and demonstrated that a "rational" canonical form is available for a suitable class of matrices, however, a counterexample was later given by Guralnick in \cite{guralnick1981similarity}. Nechaev~\cite{nechaev1986similarity} subsequently classified the similarity classes of square matrices of order $3$ over $\mathbb{Z}/p^2\mathbb{Z}.$ Kurakin~\cite{kurakin2006similarity} established a structured approach to classify similarity over SPIRs by associating to each matrix a system of polynomial ideals over a discrete valuation ring that remains invariant under similarity transformations. Avni~\cite{avni2009similarity} analyzed the similarity classes of $3 \times 3$ matrices over a local PIR, producing a classification of representatives and an explicit enumeration in the case where the residue field is finite. Prasad~\cite{prasad2015similarity} developed a theory of normal forms for similarity classes over an SPIR of length two, and described the similarity classes for square matrices with order at most $4.$ However,
	despite this progress, the classification of similarity classes over SPIRs remains incomplete in general: existing results either restrict the length of the ring, the size of the matrices, or the class of endomorphisms considered.

	On the other hand, the study of invariant submodules of an endomorphism is central to module theory and linear algebra, providing a structural framework to break down complex algebraic systems into simpler components. 
	Formally, given a module $M$ and an endomorphism $\varphi$, a submodule $N \le M$ is $\varphi$-invariant if it satisfies $\varphi(N) \subseteq N.$ When a module can be decomposed into a direct sum of invariant submodules, the endomorphism can be represented in a block diagonal matrix form, which is easier to work with. The study of invariant submodules has also proved to be useful in coding theory where, for example, polycyclic codes can be seen as invariant submodules of a certain endomorphism of $R^n$ (see \cite{radkova2009constacyclic, shi2020polycyclic, EdgarMaryamSteve, HassanTcyclic}).

	In this paper, we investigate a class of endomorphisms of finitely generated free modules over an SPIR, namely, those with annihilating polynomials that are square-free modulo the maximal ideal.

		The paper is organized as follows. Section \ref{sec:Prelim} reviews basic definitions and properties of module endomorphisms over commutative rings, alongside foundational results on local and principal ideal rings used throughout. Section \ref{sec:artinianRings} establishes useful results concerning polynomials and finitely generated free modules over Artinian rings. In Section \ref{sec:SPIR}, we describe properties of SPIRs and characterize submodules of finitely generated free modules by constructing minimal generating sets from a given basis. In Section \ref{sec:endWithSqFreeAnnihilator}, we demonstrate that over an SPIR, an endomorphism of a finitely generated free module possesses a principal annihilating ideal, provided its annihilating polynomial is square-free modulo the maximal ideal. We also prove that every free invariant submodule under such an endomorphism admits an invariant complement. Continuing the study of such endomorphisms, Section \ref{sec:invariance} investigates their invariant submodules, detailing their construction from a given module basis and calculating the total number of free invariant submodules in the finite ring case. Additionally, we introduce associated canonical matrices for these endomorphisms, allowing us to characterize matrix similarity via polynomial invariants in this context. Finally, Section \ref{sec:nonDerogInvariance} addresses the non-derogatory case where a sharper characterization of invariance is given.

	\section{Preliminaries}\label{sec:Prelim}

	Throughout this paper, all rings are assumed to be commutative with identity $1\neq 0.$	
	Let $R$ be a ring and $n\in \mathbb{N}^*.$ Let $M_{m\times n}(R)$ denote the set of $m\times n$ matrices over $R,$ $M_n(R)$ denote the $n\times n$ matrix ring over $R$ and $GL_n(R)$ denote the general linear group of degree $n$ over $R.$ Moreover, for each $R$-module $M,$  $\text{End}_R(M)$ denotes the endomorphism ring of $M$ and $\text{Aut}_R(M)$ denotes the group of automorphisms of $M.$   For a matrix $A\in M_{n}(R),$ the \textit{characteristic polynomial} of $A$ is the polynomial $\rchi_A(x)=det(xI_n-A)\in R[x].$ $\rchi_A$ is a monic polynomial of degree $n.$ Conversely, every monic polynomial of degree $n$  in $R[x]$ is the characteristic polynomial of some matrix  in $M_{n}(R).$ Namely, for $a_0,\dots, a_{n-1} \in R,$ and $f=x^n - \sum_{i=0}^{n-1}a_ix^i,$ let 
	$\displaystyle com(f)=\begin{bmatrix}
		0 & 0 & \cdots & 0 & -a_0 \\
		1 & 0 & \cdots & 0 & -a_1 \\
		0 & 1 & \cdots & 0 & -a_2 \\
		\vdots & \vdots & \ddots & \vdots & \vdots \\
		0 & 0 & \cdots & 1 & -a_{n-1}
	\end{bmatrix}$ be the \textit{companion matrix} of $f.$ Then we can verify that $\rchi_{com(f)}=f$ (see \cite{brown1993matrices} for a comprehensive treatment of matrices over commutative rings).
	
	Consider the set $R[A]=\{ r_0A^0 + \dots + r_iA^i ~|~ i\in \mathbb{N} \text{ and }r_0,\dots, r_i \in R\} \subseteq M_n(R).$ $R[A]$ is the image of the $R$-algebra morphism $\sigma_A : R[x] \longrightarrow M_n(R)$ given by the natural substitution $p(x) \longrightarrow p(A).$ The kernel of $\sigma_A$ is called the \textit{null ideal} of $A$ and we will denote it $N_A .$ We have then, $R[A]$ is isomorphic to $R[x]/N_A$ as $R$-algebras. Moreover, by the Cayley-Hamilton theorem (see \cite[Theorem 7.23]{brown1993matrices}), we have $\rchi_A \in N_A,$ so the null ideal of a nonzero matrix is nontrivial. 
	$A$ induces a natural endomorphism $R^n \longrightarrow R^n$ by acting on the canonical basis of $R^n,$ we denote it $\varphi_A.$ Conversely, for an endomorphism $\varphi \in \text{End}_R(R^n),$ we can associate a matrix in $M_n(R)$ to $\varphi$ related to the canonical basis of $R^n,$ and we denote it $M_\varphi.$
	Two matrices $M,N \in M_n(R)$ are said to be \textit{similar} if there exists an invertible matrix $P\in GL_n(R)$ such that $M=PNP^{-1}.$ Similarity induces an equivalence relation over $M_n(R),$ and if $M$ and $N$ are similar we write $M\sim N.$ It is clear that if $M \sim N,$ then $\rchi_M = \rchi_N.$ Moreover, if $\mathcal{B}$ and $\mathcal{B}'$ are two bases of $R^n,$ and $M$ and $N$ are the matrices of an endomorphism $\varphi \in \text{End}_R(R^n)$ in $\mathcal{B}$ and $\mathcal{B}'$ respectively, then $M\sim N.$ Hence we can define the characteristic polynomial of an endomorphism $\varphi$ of a finitely generated free $R$-module $F,$ to be the  characteristic polynomial of an associated matrix in some basis of the module, and we denote it $\rchi_\varphi.$ In the same manner, we define $R[\varphi]$ and $N_\varphi,$ and we have $\rchi_\varphi \in N_\varphi.$ Furthermore, we associate to $\varphi$ an $R[x]$-module denoted  $M(\varphi),$ and obtained from $F$  by defining the product of $g(x)\in R[x]$ with $\alpha \in F$ by $g(x).\alpha=g(\varphi)(\alpha).$ We call it the module associated
	with $\varphi.$
	
	If $R$ is local with maximal ideal $\mathfrak{m}$ and residue field $k,$ it is easy to see that the canonical projection $\mu : R \longrightarrow k$ induces naturally surjective morphisms $R[x] \longrightarrow k[x],$ $R^n \longrightarrow R^n/ \mathfrak{m}R^n \simeq k^n$ and $M_n(R) \longrightarrow M_n(k)$ that we all denote $\mu.$ 
	More generally, let $F$ be a finitely generated free $R$-module, $\mu$ induces an onto map $\text{End}_R(F) \longrightarrow \text{End}_k(F/\mathfrak{m}F)$ such that for each $\varphi$ in $\text{End}_R(F),$ $\mu(\varphi): F/ \mathfrak{m}F  \longrightarrow  F/ \mathfrak{m}F $ is such that for all $a\in F,$ $\mu(\varphi)(a+\mathfrak{m}F)=\varphi(a)+\mathfrak{m}F,$ and the following diagram commutes.
	
	\begin{equation}\label{eq:diagram}
		\begin{tikzcd}
			F && F \\
			\\
			{F/\mathfrak{m}F} && {F/\mathfrak{m}F}
			\arrow["\varphi", from=1-1, to=1-3]
			\arrow["\mu"{description}, from=1-1, to=3-1]
			\arrow["\mu"{description}, from=1-3, to=3-3]
			\arrow["{\mu(\varphi)}", from=3-1, to=3-3]
		\end{tikzcd}
	\end{equation}

	We conclude this preliminary section with a collection of essential results on modules over local and principal ideal rings that will be used throughout the remainder of the paper.
	\begin{theorem}(\cite[Theorem 2.3]{matsumura1989commutative})\label{thm:matsumura2.3}
		Let $(R,\mathfrak{m},k)$ be a local ring and $M$ a finitely generated $R$-module. 
		$M/\mathfrak{m}M$ is a finite-dimensional vector space over $k$, and we write $n$ for its dimension.
		\begin{itemize}
			\item[(i)] If we take a basis $\{u_1 +\mathfrak{m}M,\dots, u_n+\mathfrak{m}M\}$ of $M/\mathfrak{m}M $ over $k$, then $\{u_1,\dots,u_n\}$ is a minimal generating set of $M$;
			\item[(ii)] conversely every minimal generating set of $M$ is obtained in this way, and so has $n$	 elements.
			\item[(iii)] If $\{u_1,\dots,u_n\}$ and $\{v_1,\dots,v_n\}$ are both minimal generating sets of $M,$ and $v_i=\sum a_{ij}u_i$ with $a_{ij}\in R$ then $\text{det}(a_{ij})$ is a unit of $R,$ so that $(a_{ij})$ is an invertible matrix.
		\end{itemize}
	\end{theorem}
	
	We have then, that over a local ring $R,$ for a finitely generated $R$-module $M,$ every minimal generating set of $M$ has the same cardinality. We call the minimal number of generators of $M$ and denote it $\nu_R(M).$ An immediate consequence of Theorem \ref{thm:matsumura2.3}, is that for any two finitely generated $R$-modules $M$ and $N,$ $\nu_R(M\oplus N)=\nu_R(M)+\nu_R(N).$
	
	\begin{proposition}(\cite[Theorem 2.5.]{matsumura1989commutative})\label{prop:projectiveFreeOverLocal}
		Let $(R,\mathfrak{m})$ be a local ring, then a projective module over $R$ is free.
	\end{proposition}

	\begin{lemma}\cite[Lemma II.2]{mcdonald1978similarity}\label{lem:minimalGenSetIsLinIndep} 		
		Let $(R,\mathfrak{m})$ be a local ring and $F$ a finitely generated free $R$-module. If $\{x_1,\dots,x_n\}$ is a minimal generating set of $F,$ then it is linearly independent.
	\end{lemma}

		Using Theorem \ref{thm:matsumura2.3}, Lemma \ref{lem:minimalGenSetIsLinIndep} and
		the incomplete basis theorem for vector spaces, an immediate corollary is that we can "transfer" linear independence from the residue field over to the ring. Namely we have the following:
	
	\begin{corollary}\label{cor:independenceTransfer}
		Let $(R,\mathfrak{m},k)$ be a local ring and $F$ a finitely generated free $R$-module. Let $u_1,\dots u_r$ be elements in $F.$ If $\{u_1 +\mathfrak{m}F,\dots,u_r +\mathfrak{m}F\}$ is linearly independent in $F/\mathfrak{m}F$ then $\{u_1,\dots,u_r\}$ is  linearly independent and can be completed into a basis of $F.$ 
	\end{corollary}

	\section{Some useful results on Artinian local rings }\label{sec:artinianRings}
	
	Let $R$ be an Artinian local ring with maximal ideal $\mathfrak{m}.$ In this section, we present some results concerning polynomials over $R$ as well as finitely generated $R$-modules that will be useful throughout the paper.
	
	Let $f$ and $g$ be two polynomials in $R[x].$ $f$ is said to be \textit{regular} if it is not a zero divisor, which is the case if and only if it has a non-zero-divisor coefficient. It is said to be \textit{basic irreducible} if $\mu(f)$ is an irreducible polynomial in $k[x].$ $f$ is said to be \textit{square-free} if it is not divisible by the square of any non-constant polynomial in $R[x],$ moreover, it is called \textit{residually square-free} if $\mu(f)$ is square-free in $k[x].$ Finally, $f$ and $g$ are said to be coprime if $(f,g)=R[x].$
	As $R$ is Artinian, by Nakayama's Lemma $\mathfrak{m}$ is nilpotent, and is in fact the nilradical of $R.$ A ring  whose nilradical is a maximal ideal is called a \textit{completely primary ring}. Through a series of papers in the 50s, Snapper described the polynomial ideals over this type of rings. Namely the following result follows a characterization that appears on  \cite[Section 12]{snapperI}.

	\begin{lemma}(\cite[Section 12]{snapperI})\label{lem:regularMonic}
		\begin{itemize}
			\item[(a)] Let $f$ be a regular polynomial in $R[x].$ Then there exists a monic polynomial $g\in R[x]$ and a unit polynomial $u\in R[x]$ such that $f=ug.$
			\item[(b)] Let $I \unlhd R[x]$ be an ideal such that $\mu(I)\neq (0).$ There exists a monic polynomial $f\in I$ with least degree among the monic polynomials of $I$ such that $I=(f, N)$ where $N=I\cap \mathfrak{m}[x].$
			\end{itemize}
	\end{lemma}
	
		Using the previous result, it is easy to see that over an Artinian local ring, we have Euclidean division provided that the divisor is regular.
	\begin{lemma}(Euclidean division)\label{lem:EuclideanDiv}
		Let $f$ and $g$ be two nonzero polynomials in $R[x].$ If $g$ is regular then there exist polynomials $q$ and $r$ in $R[x]$ with $f=qg+r$ and $\deg(r)<\deg(g).$
	\end{lemma}

	\begin{proposition}\label{prop:projOfMinIsUnique}
		Let $I \unlhd R[x]$ be an ideal such that $\mu(I)\neq (0),$ and let $f,g \in I$ be two monic polynomials of least degree among the monic polynomials in $I.$ We have then $f-g\in I\cap\mathfrak{m}[x].$ In particular $\mu(f)=\mu(g).$
	\end{proposition}
	\begin{proof}
		Following Lemma \ref{lem:regularMonic}, let $I=(f_0,N)$ with $f_0$ monic, $\deg(f_0)=\deg(f)$ and $N=I\cap \mathfrak{m}[x].$ By Euclidean division, we have $f=qg+r$ with $q,r \in R[x]$ and $\deg(r) < \deg(g).$ But as $f$ and $g$ are monic with the same degree, they have the same leading term, hence $q=1$ and $f-g=r \in I.$ Let $\alpha \in R[x]$ and $\beta \in N$ such that $r=\alpha f_0 + \beta.$ If $\alpha \notin \mathfrak{m}[x],$ then the sum $r=\alpha f_0 + \beta$  has a term of degree at least $\deg(f_0),$ but as $\deg(r) < \deg(g)=\deg(f_0)$ we must have $\alpha \in \mathfrak{m}[x],$ and thus $f-g=r \in N = I\cap \mathfrak{m}[x].$		
	\end{proof}

	It is a well-known fact that an Artinian local ring is complete. We have therefore, by \cite[Theorem 8.3]{matsumura1989commutative}, $R$ is a Henselian ring, i.e. a local ring in which Hensel's lemma holds.
		
	\begin{proposition}(Hensel's lemma)\label{prop:hensel}
		Let $f(x) \in R[x]$ be a monic polynomial and let  $\mu(f)=\prod^r_{i=1} g_i,$ where $g_1 ,\dots, g_r \in k[x]$ are monic, pairwise coprime. 	Then there exist monic, pairwise coprime polynomials $f_1 , \dots, f_r \in R[x]$ such that $f=\prod^r_{i=1} f_i$	and $\mu(f_i) = g_i$ for $1 \leq i \leq r$.		
	\end{proposition}

	Although we do not have unique factorization in $R[x]$ in general, the next proposition shows that residually square-free polynomials behave favorably in this regard.
	
	\begin{proposition}\label{prop:sqFreeUniqIrrFacto}
		A residually square-free monic polynomial in $R[x]$ factors uniquely into monic, pairwise coprime basic irreducible polynomials. 
	\end{proposition}
	\begin{proof}
		The existence of such a decomposition follows directly from Hensel's lemma. The proof of the uniqueness is similar to that given in \cite[Theorem XIII.11]{mcdonald1974finite} for finite local rings.
	\end{proof}
	
	The next lemma is given in \cite[Lemma 2.8]{NortonSalagean2} for the case of finite chain rings. The proof is the same for Artinian local rings.
	\begin{lemma}\label{lem:coprimeIffCoprime}  Let $f$ and $g$ be two regular
		polynomials in $R[x]$. Then $f$ and $g$ are coprime  if and only if $\mu(f)$ and $\mu(g)$ are coprime polynomials in $k[x]$.
	\end{lemma}

	While no additional condition is required on a local ring  for the generating property of a subset to be transferred via $\mu$ and vice-versa as seen in Theorem \ref{thm:matsumura2.3}, the transfer of linear independence needs more. It turns out, as we will see in Proposition \ref{prop:steinitz}, that for a local ring to have this linear independence transfer property, it needs to satisfy a sort of incomplete basis property.	
	A ring $S$ is said to be \textit{weakly semi-Steinitz},  if every linearly independent finite subset of a finitely generated free $S$-module $F$ can be extended to a basis of $F.$
	See \cite{nashier1991steinitz,chwe1970extension} for more details about weakly semi-Steinitz and  Steinitz rings in general. In particular, \cite[Corollary 2.5.]{nashier1991steinitz} states that a commutative Noetherian ring is weakly semi-Steinitz if and only if every non-zero-divisor is a unit, hence it is clear that an Artinian local ring is weakly semi-Steinitz. As a consequence, we have the following straightforward properties.
	
	\begin{proposition}\label{prop:steinitz}
		Let $F$ be a finitely generated free $R$-module.  We have the following:
		\begin{itemize}
			\item[(i)] If $\{w_1,\dots,w_r\}$ is a linearly independent subset of $F,$ then $\{w_1+\mathfrak{m}F,\dots,w_r+\mathfrak{m}F\}$ is linearly independent in $F/\mathfrak{m}F.$
			\item[(ii)] Each linearly independent subset of $F$ with cardinality $\operatorname{rank}_R(F)$ is a basis of $F.$
			\item[(iii)] For any free submodule $V$ of $F,$ $V$ is a direct summand of $F.$
		\end{itemize}
	\end{proposition}
	
		\begin{remark}
		Let $F$ be a finitely generated free $R$-module and $M$ a free submodule of $F.$ $\mu(M)$ may refer to two  objects: the image of $M$ under the projection $\mu:F \longrightarrow F/\mathfrak{m}F,$ or $M/\mathfrak{m}M$ the image of the induced projection $\mu:M \longrightarrow M/\mathfrak{m}M$ as in diagram \ref{eq:diagram}. However, as $R$ is weakly semi-Steinitz, from  Proposition \ref{prop:steinitz}(iii) $M$ is a direct summand of $F,$ hence the natural map $M/\mathfrak{m}M \rightarrow F/\mathfrak{m}F$ is injective and the two objects coincide.
	\end{remark}

		\begin{proposition}\label{prop:charaDividesChara}
		Let $F$ be a finitely generated free $R$-module and $\varphi \in \text{End}_R(F).$ Let $W\le F$ be a free submodule invariant under $\varphi.$ Then $\rchi_{\varphi_{\restriction W}}$ divides $\rchi_\varphi.$ In particular, if $F=W\oplus W'$ such that $W'$ is also $\varphi$-invariant, then  $\rchi_\varphi=\rchi_{\varphi_{\restriction W}}\rchi_{\varphi_{\restriction W'}}.$
		
	\end{proposition}
	
	\begin{proof}
		As $R$ is weakly semi-Steinitz, from  Proposition \ref{prop:steinitz}(iii), $W$ is a direct summand of $F.$ Let $F=W\oplus W'$ where $W'$ is some free submodule of $F.$ Let $\mathcal{B}$ and $\mathcal{B'}$ be bases of $W$ and $W'$ respectively. As $W$ is $\varphi$-invariant, $\varphi$ is represented in the ordered basis $\mathcal{B}\cup \mathcal{B'}$ as a matrix $A=\begin{pmatrix}
			B & C \\
			0 & D
		\end{pmatrix},$ with $B$ the matrix of $\varphi_{\restriction W}$ in $\mathcal{B}.$ It follows that $det(xI-A)=det(xI-B)det(xI-D),$ and thus $\rchi_{\varphi_{\restriction W}}$ divides $\rchi_\varphi$ (note that $I$ denotes different identity matrices in the previous by abuse of notation). Moreover, if $W'$ is also $\varphi$-invariant, then $C=0$ and $D$ is the matrix of $\varphi_{\restriction W'}$ in $\mathcal{B'}.$ Hence $det(xI-D)=\rchi_{\varphi_{\restriction W'}}$   and $\rchi_\varphi=\rchi_{\varphi_{\restriction W}}\rchi_{\varphi_{\restriction W'}}.$		
	\end{proof}

	\section{Finitely generated free modules over an SPIR}\label{sec:SPIR}
	
	In this section,  we describe the structure of submodules of a finitely generated free module over an Artinian local principal ideal ring. We start with some useful results concerning the structure of these rings. For a detailed discussion about SPIRs see \cite[Volume 1, Chapter 4, Section 15]{zariski1960commutative}.	
	Let $R$ be an SPIR with maximal ideal $\mathfrak{m},$ and let $\gamma\in R$ such that $\mathfrak{m}=(\gamma).$ As $R$ is Artinian,  the ideal $\mathfrak{m}$ is nilpotent. Let $\nu\in \mathbb{N}^*$ denote the nilpotency index of $\mathfrak{m}.$ The ideals of $R$ form a strictly increasing chain 
	\begin{equation}\label{eq:idealChain}
		(0)=\mathfrak{m}^\nu \subsetneq\mathfrak{m}^{\nu-1}\subsetneq\dots \subsetneq \mathfrak{m} \subsetneq R\,.
	\end{equation}

	\begin{lemma}\label{lem:uniqueExpressionSPIR}
		Let $F$ be a finitely generated free $R$-module. For any element $w \in F \setminus \{0\},$ there is a unique integer $s$ and an element $w'\in F\setminus \mathfrak{m}F$ such that $w=\gamma^sw',$ where $0\leq s\leq \nu-1.$
	\end{lemma}
	\begin{proof}
		Let $n\in \mathbb{N}$ such that $F\simeq_R R^n,$ and let $(w_1,\dots,w_n) \in R^n\setminus \{0\}$ be the vector corresponding to $w$ in $R^n.$
		For each $1\leq i\leq n$ such that $w_i\neq 0,$  by \cite[Proposition 2.2]{norton2001strong}, there exists a unique $s_i$ and a unit $u_i\in R$ such that $w_i=u_i\gamma^{s_i}.$ The result follows then by  taking $s$ to be $\min_{1\leq i\leq n}\{s_i\}.$
	\end{proof}

	\begin{lemma}\label{lem:uniqueExpressionNorton2000} 
		Let $F$ be a finitely generated free $R$-module.  Then for each $w\in F,$ if $1\leq i\leq j\leq \nu$ and $\gamma^iw\in \mathfrak{m}^jF$ then $w\in \mathfrak{m}^{j-i}F.$ 
	\end{lemma}
	\begin{proof}
	Follows as in the case of finite chain rings presented in \cite[Section 2.3]{NortonSalagean2}.
	\end{proof}

	The next result was given in  \cite[3.1 Lemma]{1317117} for the case of finite chain rings. We present it here for SPIRs with a shorter proof.
	\begin{proposition}\label{prop:quotientByIrrIsSPIR}
		
		Let $f\in R[x]$ be a basic irreducible polynomial. Then $R[x]/(f)$ is an SPIR, with nilpotency index $\nu,$ and ideals 
		\begin{equation}
			(0)\subsetneq(\gamma^{\nu-1} +(f))\subsetneq\dots \subsetneq (\gamma + (f))\subsetneq R[x]/(f)\,.
		\end{equation}
	\end{proposition}
	\begin{proof}
		By Lemma \ref{lem:regularMonic}, we may assume $f$ is monic \textit{w.l.o.g}.
		Let $I \unlhd R[x]$ be an ideal containing $f,$ and let $f_1,\dots,f_q$ be such that $I=(f_1,\dots,f_q).$ For all $i,$ let $r_i\in R[x]$ be the rest of the Euclidean division of $f_i$ by $f.$ If $r_i=0$ let $\lambda_i=\nu.$ Otherwise, let  $\lambda_i \in \mathbb{N}$ and $r^*_i \in R[x]$ regular such that $r_i=\gamma^{\lambda_i}r^*_i$ and $\text{deg}(r_i)=\text{deg}(r^*_i).$ As $\mu(f)$ is basic irreducible and $\deg(\mu(r^*_i))<\deg(f),$ $\mu(r^*_i)$ and $\mu(f)$ are coprime, hence by Lemma \ref{lem:coprimeIffCoprime}, $r^*_i$ and $f$ are coprime, which means that $r^*_i$ is a unit modulo $f.$ Hence, $I/(f)=(\gamma^{\lambda_1}+(f),\dots,\gamma^{\lambda_q}+(f))=(\gamma^{min\{\lambda_i\}_i}+(f)).$ Furthermore, as $R$ can naturally be embedded in $R[x]/(f),$ the nilpotency index cannot be less than $\nu,$ which implies that it equals $\nu$ and that the ideal chain is strict.
	\end{proof}
	
	As an immediate corollary, we get the following result which is well known in the context of finite chain rings (see for example \cite[Corollary 4.6]{NortonSalagean2} or \cite[3.7 Corollary]{1317117}).
	\begin{corollary}\label{cor:quotientRingIsPIR}		
		Let $f\in R[x]$ be a residually square-free polynomial. Then $R[x]/(f)$ is a principal ideal ring.
	\end{corollary}

	We now investigate the structure of the submodules of a finitely generated free $R$-module. The next lemma is an adaptation  of  \cite[Lemma 1, p. 246]{zariski1960commutative} to the SPIR case.
	
		\begin{lemma}(\cite[Lemma 1, p. 246]{zariski1960commutative})\label{lem:numMinGenSubModSPIR} Let $M$ be a finitely generated $R$-module and $N$ a submodule of $M.$ Then we have $N$ is finitely generated and $\nu_R(N)\leq \nu_R(M).$
	\end{lemma}

	A matrix $D=(d_{i,j})\in M_{m\times n}(R)$ is called a diagonal matrix if
	$d_{i,j}=0$ whenever $i\neq j.$ Such a matrix can be expressed as
	$D=\operatorname{diag}(d_1,\ldots,d_r),
	$
	where \(r=\min\{m,n\}\) and \(d_i=d_{i,i}\) for
	\(i=1,\ldots,r\).	
	According to \cite[Theorem 15.24]{brown1993matrices}, for every matrix
	$A\in M_{m\times n}(R)$, there exist invertible matrices \(P\) and \(Q\)
	and a diagonal matrix
	$D=\operatorname{diag}(d_1,d_2,\ldots,d_r)$	such that $A=PDQ,$
	where the diagonal entries satisfy the divisibility relations
	$d_1\mid d_2\mid\cdots\mid d_r.$
	The elements \(d_1,d_2,\ldots,d_r\) are uniquely determined up to
	associates, and \(D\) is called the Smith normal form of \(A\). A version of the following theorem is given in \cite[Proposition 3.2]{kamche2019rank} for the case of finite PIRs.
	
	\begin{theorem}\label{thm:subModEnumFromAGivenBasis}
		Let $F$ be a finitely generated free module of rank $n$ and let $\mathcal{B}=\{u_1,\dots,u_n\}$ be a basis of $F.$ For each integer $1\leq m\leq n,$ denote $\mathcal{M}_m$ the set of submodules of $F$ with a minimal number of generators equals to $m.$ Then we have
		\begin{align*}
			\mathcal{M}_m=\left\{\langle\gamma^{k_1}\psi(u_1),\dots,\gamma^{k_m}\psi(u_m)\rangle ~|~ 0\leq k_1,\dots,k_m\leq \nu-1  \text{ and } \psi\in \operatorname{Aut}_R(F)\right\}.
		\end{align*} 		
	\end{theorem}
	\begin{proof}
		For each automorphism $\psi \in  \operatorname{Aut}_R(F),$ $\{\psi(u_1),\dots,\psi(u_m)\}$ is a linearly independent set. Let $k_1,\dots,k_m \in\{0,\dots,\nu-1\}.$ Denote $\mathcal{E}$ the set $\{\gamma^{k_1}\psi(u_1),\dots,\gamma^{k_m}\psi(u_m)\}$ and $M=\langle \mathcal{E} \rangle.$  Suppose $\mathcal{E}$ is not a minimal generating set of $M.$		
		 \textit{W.l.o.g} we can then suppose that there exist $\alpha_2,\dots,\alpha_m \in R$ such that $\gamma^{k_1}\psi(u_1)=\sum_{i=2}^{m}$$\alpha_i \gamma^{k_i}\psi(u_i),$ but this implies that $\gamma^{k_1}=\alpha_2 \gamma^{k_2}=\dots=\alpha_m \gamma^{k_m}=0$ which cannot be as $k_1<\nu.$ Hence $M \in \mathcal{M}_m.$
		 
		 Conversely, let $M\in\mathcal{M}_m$ with a minimal generating set $\{v_1,\dots,v_m\}.$ Let $A\in M_{m\times n}(R)$ be the matrix whose rows are the coordinates of $v_1,\dots,v_m,$ in the basis $\mathcal{B}.$ By \cite[Theorem 15.24]{brown1993matrices}, let $D=\operatorname{diag}(d_1,d_2,\ldots,d_m)$ be the Smith normal from of $A,$ and $P\in GL_m(R)$ and $Q\in GL_n(R)$ such that $A=PDQ.$ Right-multiplying by $Q$ corresponds to changing the basis of $F,$ while left-multiplying by $P$ corresponds to changing the minimal generating set of $M.$
		 Hence there exists an automorphism $\psi'\in Aut_R(F)$ such that $\{d_1\psi'(u_1),\dots,d_m \psi'(u_m)\}$ is a minimal generating set of $M.$ But as $\nu_R(M)=m,$ that means that no $d_i$ is zero, and thus by Lemma \ref{lem:uniqueExpressionSPIR}, for all $1\leq i\leq m,$ there exist $0\leq k_i\leq \nu-1$ and a unit $d'_i\in R\setminus \mathfrak{m}$ such that $d_i=\gamma^{k_i}d'_i.$ Finally, it follows immediately that there exists an automorphism $\psi\in Aut_R(F)$ such that $\{\gamma^{k_1}\psi(u_1),\dots,\gamma^{k_m}\psi(u_m)\}$ is a minimal generating set of $M.$	\end{proof}

	For a type of submodules of a finitely generated free module, we can say more about their minimal generating sets.

	\begin{proposition}\label{prop:moduleChainHasSameMinGenNum}
		Let $F$ be a finitely generated free $R$-module with basis $\{u_1,\dots,u_n\}.$
		\begin{itemize}
			\item[(a)] For all $0\leq i\leq \nu-1,$ $\mathfrak{m}^iF/\mathfrak{m}^{i+1}F$ and $F/\mathfrak{m}F$ are isomorphic both as $R$-modules and as $k$-vector space.
			\item[(b)]  For all $0\leq i\leq \nu-1,$ the $R$-module $\mathfrak{m}^iF$ has a minimal generating set $\{\gamma^iu_1,\dots,\gamma^iu_n\}.$ In particular, we have 
			\begin{equation}
				\nu_R(\mathfrak{m}^{\nu-1}F)=\nu_R(\mathfrak{m}^{\nu-2}F)=\dots=\nu_R(\mathfrak{m}F)=\text{rank}_R(F).
			\end{equation}
		\end{itemize}
	\end{proposition}
	\begin{proof}
		\textit{(a)} Fix $0\leq i\leq \nu-1,$ and let the map $\theta : \mathfrak{m}^iF/\mathfrak{m}^{i+1}F \longrightarrow F/\mathfrak{m}F $ be such that for all $w\in F,$ $\theta(\gamma^iw+\mathfrak{m}^{i+1}F)=w+\mathfrak{m}F.$ $\theta$ is a well-defined $R$-linear map. Indeed, the linearity can be checked easily and for all $w_1,w_2\in F$ such that $\gamma^iw_1-\gamma^iw_2 \in \mathfrak{m}^{i+1}F,$ we have $\gamma^i(w_1-w_2)\in \mathfrak{m}^{i+1}F,$ hence by Lemma \ref{lem:uniqueExpressionNorton2000}, $w_1-w_2\in \mathfrak{m} F,$ and thus $\theta(\gamma^iw_1+\mathfrak{m}^{i+1}F)=\theta(\gamma^iw_2+\mathfrak{m}^{i+1}F).$
		Moreover, $\theta$ is clearly onto, and for all $w\in F,$ if $w\in \mathfrak{m}F$ then $\gamma^iw\in \mathfrak{m}^{i+1}F,$ hence $\theta$ is one-to-one and thus it is an $R$-module isomorphism. 	
		On the other hand, endowing $\mathfrak{m}^iF/\mathfrak{m}^{i+1}F$ with the $k$-vector space structure such that for all $\alpha \in k,$ if $\beta \in R$ is such that $\mu(\beta)=\alpha$ then for each $w\in \mathfrak{m}^iF/\mathfrak{m}^{i+1}F,$ $\alpha w=\beta w,$ we can easily verify that $\theta$ is a $k$-linear map w.r.t this structure, hence  $\theta$ is also an isomorphism between  $k$-vector spaces.
		
		\textit{(b)} By Theorem \ref{thm:matsumura2.3}, $\{u_1+\mathfrak{m}F,\dots,u_n+\mathfrak{m}F\}$ is a $k$-basis of $F/\mathfrak{m}F.$ Fix $0\leq i\leq\nu-1.$ From the proof of \textit{(a)} above, we see that  $\{\gamma^iu_1 + \mathfrak{m}^{i+1}F,\dots,\gamma^iu_n + \mathfrak{m}^{i+1}F\}$ is a basis of the $k$-vector space $\mathfrak{m}^iF/\mathfrak{m}^{i+1}F,$ hence by Theorem \ref{thm:matsumura2.3}, $\{\gamma^iu_1,\dots,\gamma^iu_n\}$ is a minimal generating set of $\mathfrak{m}^iF,$ and $\nu_R(\mathfrak{m}^iF)=\operatorname{rank}_R(F).$
			\end{proof}

	\begin{remark}
		Note that unlike the case of modules of the form $\mathfrak{m}^iF$ as in  Proposition \ref{prop:moduleChainHasSameMinGenNum}, where each basis of $F$ can be turned into a minimal generating set of $\mathfrak{m}^iF,$ for a general submodule $M\le F$ with $\nu_R(M)=n,$ Theorem \ref{thm:subModEnumFromAGivenBasis} only asserts that we can obtain a minimal generating set of $M$ from a given basis of $F.$ Take $R=\mathbb{Z}/8\mathbb{Z}$ for example, and let $\{e_1,e_2\}$ be the canonical basis of $R^2.$ Let $M=\langle 2(e_1+e_2),4e_2\rangle\le R^2$. We can easily verify that for all $0\leq i,j\leq 2,$ $M\neq \langle 2^ie_1, 2^je_2 \rangle.$
	\end{remark}

	\section{Endomorphisms with residually square-free annihilator polynomials }\label{sec:endWithSqFreeAnnihilator}
	
	Throughout the remainder of the paper, $R$ is an SPIR with maximal ideal $\mathfrak{m}=\gamma R,$ nilpotency index $\nu,$ and residue field $k=R/\mathfrak{m},$ $F$ is a finitely generated free $R$-module of rank $n\in \mathbb{N},$ and $\varphi\in\text{End}_R(F).$

		 Let $\lambda_\varphi \in N_\varphi$ denote a monic polynomial of least degree among the monic polynomials in 
		$N_\varphi.$ Following Proposition \ref{prop:projOfMinIsUnique}, we see that each monic polynomial of least degree in $N_\varphi$ has the same projection $\mu(\lambda_\varphi)$ in $k[x].$ It is easy to see that $\mu(\rchi_\varphi)=\rchi_{\mu(\varphi)}.$ On the other hand, denoting $\lambda_{\mu(\varphi)}$ the minimal polynomial of $\mu(\varphi),$ we have $\lambda_{\mu(\varphi)}$ divides $\mu(\lambda_\varphi)$ as the latter annihilates $\mu(\varphi),$ however, we do not have in general that  $\mu(\lambda_\varphi)=\lambda_{\mu(\varphi)},$ and this makes working with $\varphi$ quite difficult. To deal with this problem, McDonald introduced the concept of regular endomorphisms over SPIRs in \cite{mcdonald1978similarity} which was generalized later in \cite{guralnick1981similarity} for an arbitrary commutative local rings.
		 $\varphi$ is said to be \textit{regular} if $R[\varphi]$ is a free $R$-submodule of $\text{End}_R(F).$ The strength of this notion in the SPIR case, lies in the following proposition.

			\begin{proposition}(\cite[Theorems II.3 and II.4 ]{mcdonald1978similarity})\label{prop:regularCharacterization}
				The Following are equivalent: 
				\begin{itemize}
					\item[(i)] $\varphi$ is regular,
					\item[(ii)] $N_\varphi$ is a principal ideal generated by a monic polynomial,
					\item[(iii)] $\mu(\lambda_\varphi)=\lambda_{\mu(\varphi)}.$
					
				\end{itemize}
						
		\end{proposition}	
		
		It follows that if $\varphi$ is  regular, then $N_\varphi$ has a unique monic polynomial of least degree. We call it the \textit{minimal polynomial} of $\varphi$ and denote it $\lambda_\varphi.$ And so in this case we have $N_\varphi=(\lambda_\varphi).$
			
		\begin{proposition}\label{prop:kernelsProp}
				Suppose there exists a residually square-free polynomial $f\in N_\varphi,$ and let $f=\prod_{1\leq i\leq r}f_i$ be a decomposition of $f$ into  pairwise coprime basic irreducible polynomials. Denote $V_i=ker(f_i(\varphi))$ and $\overline{V_i}=ker(\mu(f_i)(\mu(\varphi)))$ for $i \in \{1,\dots,r\}.$ We have the following:
				\begin{itemize}
					\item[(i)] For each $i\in\{1,\dots,r\},$ $V_i$ is free and invariant under $\varphi,$ and $F=\bigoplus^{r}_{i=1} V_i.$ Moreover, if $N_\varphi=(f)$ then for $i=1,\dots,r,$  $\varphi_{|V_i}$ is regular with $N_{\varphi_{|V_i}}=(f_i).$
					\item[(ii)] For each $i\in\{1,\dots,r\},$ $\overline{V_i}=\mu(V_i)$ and $\operatorname{rank}_R(V_i)=\operatorname{dim}(\overline{V_i}).$ In particular, denoting $k_i$ the multiplicity  of $\mu(f_i)$ in $\rchi_{\mu(\varphi)},$ we have $\operatorname{rank}_R(V_i)=k_i\deg(f_i).$
				\end{itemize}
				
			\end{proposition}
		
			\begin{proof}
				\textit{(i)} The proof that $F=\bigoplus^{r}_{i=1} V_i$ with each $V_i$ invariant under $\varphi$ is the same as the case of vector spaces, see \cite[Theorem 12, Chapter 6]{hoffmann1971linear} for example. Moreover, from Proposition \ref{prop:projectiveFreeOverLocal}, each $V_i$ is free.
				Now suppose that $N_\varphi=(f).$ We prove the result for $r=2$ and the general case follows immediately by induction. From Lemma \ref{lem:regularMonic} we can suppose w.l.o.g that $f$ is monic as well as $f_1$ and $f_2.$ For $i=1,2,$ denote $\varphi_i=\varphi_{\restriction V_i}.$		
				We have obviously that  $f_i\in N_{\varphi_i}.$ Now suppose there exists $g\in N_{\varphi_1} $ such that $\deg(g)<\deg(f_1).$ As $F=V_1\oplus V_2,$ for each $v\in F$ there exists $v_1\in V_1$ and $v_2\in V_2$ such that $v=v_1+v_2,$ so we have $(gf_2)(\varphi)(v)=(gf_2)(\varphi_1)(v_1)+(gf_2)(\varphi_2)(v_2)=f_2(\varphi_1)(g(\varphi_1)(v_1))+g(\varphi_2)(f_2(\varphi_2)(v_2))=0.$
				Hence $gf_2 \in N_\varphi=(f),$ but as $\deg(g)<\deg(f_1),$ $\deg(gf_2)<\deg(f)$ which implies that $gf_2=0,$ and as $f_2$ is monic this means that $g=0.$ $f_1$ is then of minimal degree among the nonzero polynomials of $N_{\varphi_1},$ hence by Euclidean division, as $f_1$ is monic, $N_{\varphi_1}=(f_1).$ The case of $\varphi_2$ follows by symmetry.
								
				\textit{(ii)}  Same proof as the one given for finite chain rings in \cite [3.5. Proposition]{nevcaev1973finite}.		\end{proof}

			A standard theorem in linear algebra states that the minimal polynomial of a block diagonal matrix is the least common multiple of the minimal polynomials of its diagonal blocks. We adapt this result to our context as follows.

		\begin{proposition}\label{prop:sqRegularityIsLocal}
		Suppose there exists a residually square-free polynomial $f$ in $N_\varphi$. If $F$ is a direct sum of $\varphi$-invariant submodules such that the restriction of $\varphi$ to each summand is a regular endomorphism, then $\varphi$ is regular.
		\end{proposition}
		\begin{proof}
			As $\mu(f)$ is square-free, by Lemma \ref{lem:regularMonic} we can take $f$ to be monic \textit{w.l.o.g}. Let $f=\prod_{i=1}^{r}f_i$ be the unique factorization of $f$ into monic pairwise coprime basic irreducible polynomials.  As $R$ is local, each summand of the direct sum is free, hence there exists a basis $\mathcal{B}$ of $F$ relative to which the matrix of $\varphi$ is of the form $diag(M_1,\dots,M_q)$ for some regular matrices $M_1,\dots,M_q.$ It is easy to see that $\mu(\varphi)$ is represented by $diag(\mu(M_1),\dots,\mu(M_q))$ in the ordered basis $\mu(\mathcal{B})$ of $F/\mathfrak{m}F,$ hence the minimal polynomial of $\mu(\varphi)$ is $\text{l.c.m}(\lambda_{\mu(M_1)},\dots,\lambda_{\mu(M_q)}).$ For each $i\in \{1,\dots,q\},$ as $M_i$ is regular and annihilated by $f,$ we have $\mu(\lambda_{M_i})=\lambda_{\mu(M_i)}$ and  $\lambda_{M_i}$ divides $f,$ so by the uniqueness of the irreducible decomposition of Proposition \ref{prop:sqFreeUniqIrrFacto}, there exists a subset $J_i\subset \{1,\dots,r\}$ such that $\lambda_{M_i}=\prod_{j\in J_i}f_j.$ Let $J=\cup_{1\leq i\leq q}J_i$ and denote  $g=\prod_{i\in J}f_i.$ We have $g\in N_\varphi$ and $\mu(g)=\text{l.c.m}(\mu(\lambda_{M_1}),\dots,\mu(\lambda_{M_q}))=\text{l.c.m}(\lambda_{\mu(M_1)},\dots,\lambda_{\mu(M_q)})=\lambda_{\mu(\varphi)},$ hence by Proposition \ref{prop:regularCharacterization}, $\varphi$ is regular. 
		\end{proof}

		Next, we consider the case where the characteristic polynomial is residually square-free.

		\begin{proposition}\label{prop:charaSqFreeImpliesRegular}
			If $\mu(\rchi_\varphi)$ is square-free, then $\varphi$ is regular and $N_\varphi=(\rchi_\varphi).$ Furthermore, there exists an element $w\in F$ such that $\mathcal{B}=\{w,\varphi(w),\dots,\varphi^{n-1}(w)\}$ is a basis of $F$ and  $\varphi$ is represented in $\mathcal{B}$ by $com(\rchi_\varphi)$.			
		\end{proposition}
		\begin{proof}
			As $\mu(\rchi_\varphi)=\rchi_{\mu(\varphi)},$ if $\mu(\rchi_\varphi)$ is square-free, then  $\rchi_{\mu(\varphi)}=\lambda_{\mu(\varphi)},$ which implies by Proposition \ref{prop:regularCharacterization}(iii) that $\varphi$ is regular and $N_\varphi=(\rchi_\varphi)$. Moreover, as $\rchi_{\mu(\varphi)}=\lambda_{\mu(\varphi)},$ there exists an element $w\in F$ such that $\{w+\mathfrak{m}F,\mu(\varphi)(w+\mathfrak{m}F),\dots,\mu(\varphi)^{n-1}(w+\mathfrak{m}F)\}$ is a basis of $F/\mathfrak{m}F$ (see \cite[page 237, Corollary]{hoffmann1971linear}). Hence, as $\mu(\varphi)^i(w+\mathfrak{m}F)=\varphi^i(w)+\mathfrak{m}F$ for all $i,$ by Corollary \ref{cor:independenceTransfer} we have that $\{w,\varphi(w),\dots,\varphi^{n-1}(w)\}$ is linearly independent in $F,$ which implies by Proposition  \ref{prop:steinitz}(ii) that it is a basis of $F,$  and it is easy to check that $\varphi$ is represented by $com(\rchi_\varphi)$ in it.	
		\end{proof}

		 In \cite[Theorem III.2]{mcdonald1978similarity}, it was claimed that each regular endomorphism of $F$ can be represented by a matrix that splits as a sum of companion matrices  $\text{diag}(C_1,\dots,C_r)$ for some companion matrices $C_1,\dots,C_r.$ This was soon after shown to be false, (see e.g.  \cite{guralnick1981similarity,nechaev1986similarity}). However, the next proposition shows that it holds when the minimal polynomial of the endomorphism is residually square-free.

		\begin{proposition}\label{prop:squareFreeAnnihilatorImpliesNormal}
		 If $\varphi$ has a residually square-free annihilating polynomial, then a basis of $F$ may be chosen wherein $\varphi$ is represented by a matrix that splits as a sum of companion matrices.
		\end{proposition}
		\begin{proof}
			From \cite[Theorem 1]{nechaev1986similarity}, we have that $\varphi$ may be represented by a matrix that splits as a sum of companion matrices if and only if the $R[x]$-module $M(\varphi)$ decomposes into a direct sum of cyclic modules, hence all we need to prove is that $M(\varphi)$ has such decomposition. 
			Let $f\in N_\varphi$ be such that $\mu(f)$ is square-free. As $(f)\subseteq \operatorname{Ann}_{R[x]}(M(\varphi)),$ $M(\varphi)$ can be considered as a module over $R[x]/(f)$ in a  natural way, moreover, it decomposes as a direct sum of cyclic modules as an $R[x]$-module if and only if it does as an $R[x]/(f)$-module. But as $\mu(f)$ is square-free, $R[x]/(f)$ is a PIR by Corollary \ref{cor:quotientRingIsPIR}. Hence by the structure theorem for finitely generated modules over a PIR \cite[Theorem 15.33]{brown1993matrices}, as an $R[x]/(f)$-module, $M(\varphi)$ is a direct sum of cyclic modules.
		\end{proof}
		
			Combining the previous propositions yields the following important result. 
		\begin{theorem}\label{thm:sqFreeImplyRegular}
			If there exists $f\in N_\varphi$ such that $\mu(f)$ is square-free then $\varphi$ is regular, and $N_\varphi=(g)$ for some monic polynomial $g\in R[x]$ that divides $f.$
		\end{theorem}
		\begin{proof}
			Following Proposition \ref{prop:squareFreeAnnihilatorImpliesNormal}, let $\mathcal{B}$ be a basis of $F$ relative to which the matrix of $\varphi$ is 
			$\text{diag}(C_1,\dots,C_r)$ for some companion matrices $C_1,\dots,C_r.$ Fix an integer $1\leq i\leq r.$ It is easy to see that $f\in N_{C_i},$ hence $\mu(f)$ annihilates $\mu(C_i).$ But as $C_i$ is a companion matrix, so is $\mu(C_i),$ hence $\rchi_{\mu(C_i)}=\lambda_{\mu(C_i)}$ which implies that $\rchi_{\mu(C_i)}$ divides $\mu(f).$ We have then $\rchi_{\mu(C_i)}$ is square-free, but as $\rchi_{\mu(C_i)}=\mu(\rchi_{C_i}),$ using Proposition \ref{prop:charaSqFreeImpliesRegular} we see that $C_i$ is regular. Finally, by Proposition \ref{prop:sqRegularityIsLocal}, $\varphi$ is also regular.
		\end{proof}
		
		\begin{corollary}\label{cor:sqFreeAnnPolImpliesSqFreeMinPoly}
			$\varphi$ has a residually square-free annihilating polynomial if and only if $\mu(\lambda_\varphi)$ is square-free.
		\end{corollary}
		\begin{proof}
			Suppose there exists $f\in N_\varphi$ such that $\mu(f)$ is square-free. We have then by Theorem \ref{thm:sqFreeImplyRegular}, $\varphi$ is regular. Hence  $N_\varphi=(\lambda_\varphi)$ and $\lambda_\varphi$ divides $f.$ We have then $\mu(\lambda_\varphi)$ divides $\mu(f)$ and thus it is square-free. The converse is immediate, since $\lambda_\varphi\in N_\varphi.$ 
		\end{proof}
	
		\begin{remark}
			The condition that not only the annihilating polynomial is square-free but also its   projection modulo $\mathfrak{m},$ is necessary for the theorem to hold. Over the SPIR $R=\mathbb{Z}/8\mathbb{Z}$ for example, the matrix $A=\begin{pmatrix}
					2 & 0   \\
					0 & 4 
				\end{pmatrix}$ is annihilated by the polynomial $(x-2)(x-4)$ which is easily verified to be square-free, however it is not regular. Indeed, as $\mu(A)=0,$ $\lambda_{\mu(A)}=x,$ so for $A$ to be regular we must have that $A$ is annihilated by a polynomial of the form $x-\alpha$ with $\alpha$ a zero-divisor in $R,$ and it is easy to see that it is not the case.
		\end{remark}
		
		\begin{remark}(Converse of Proposition \ref{prop:sqRegularityIsLocal})\label{rk:localRegularityConverse}
			Let $W$ be a free $\varphi$-invariant submodule of $F.$ An annihilating polynomial of $\varphi$  obviously annihilates $\varphi_{\restriction W}.$ Hence if $\varphi$ has a  residually square-free annihilating polynomial, then so does $\varphi_{\restriction W},$ and by Theorem \ref{thm:sqFreeImplyRegular}, $\varphi_{\restriction W}$ is also regular. 	\end{remark}

		\begin{notation}	
		For the rest of this paper, we assume that $\varphi$ has a residually square-free annihilating polynomial. In light of Theorem \ref{thm:sqFreeImplyRegular} and Corollary \ref{cor:sqFreeAnnPolImpliesSqFreeMinPoly}, this means that $\lambda_\varphi$ is unique and residually square-free,
		so by Proposition \ref{prop:sqFreeUniqIrrFacto} let $\lambda_\varphi=\prod_{i=1}^{r}f_i$ be the factorization of $\lambda_\varphi$ into monic basic irreducible pairwise coprime polynomials, and let $m_1,\dots,m_r \in \mathbb{N}^*$ be such that $\rchi_{\mu(\varphi)}=\prod_{i=1}^{r}\mu(f_i)^{m_i}.$ Moreover, for all $1\leq i\leq r,$ denote $V_i=ker(f_i(\varphi)),$ and $\varphi_i$ the restriction of $\varphi$ to $V_i.$
		\end{notation}

		Recall that a linear operator $T : V \rightarrow V$ on a vector space $V$ is called semisimple, if every $T$-invariant subspace has a $T$-invariant complement. In light of this and Proposition \ref{prop:projectiveFreeOverLocal}, we give the following definition of a semisimple endomorphism of $F.$  
		
		\begin{definition}
			An endomorphism $\psi\in \text{End}_R(F)$ is said to be \textit{semisimple} if every free $\psi$-invariant submodule of $F$ has a $\psi$-invariant complement.
		\end{definition}

		We conclude this section by showing that $\varphi$ is a semisimple endomorphism. We start with the following two propositions which will be useful throughout the remainder of the paper.

		\begin{proposition}\label{prop:invariantDirectSumWhenMinSqFree}	
			For each $\varphi$-invariant submodule $W$ of $F,$ we have \begin{itemize}
				\item[(i)] $W=\bigoplus_{i=1}^{r}W\cap V_i,$ 
				\item[(ii)] for all $1\leq i\leq r,$  $W\cap V_i$ is invariant under $\varphi,$
				\item[(iii)] $W$ is free if and only if $W\cap V_i$ is free for all $i.$
			\end{itemize}
		\end{proposition}
		\begin{proof}
			\textit{(i)} Let $w\in W.$ From Proposition \ref{prop:kernelsProp}, we have that $F=\bigoplus_{1\leq i\leq r}V_i,$ hence, for all $i,$ there exists a unique $w_i\in V_i$ such that $w=w_1+\dots+w_r.$ For all $1\leq i\leq r,$ let $\hat{f_i}=\prod_{\substack{1\leq j\leq r \\ j\neq i}}f_j .$ As the $f_i$'s are pairwise coprime, $f_i$ and $\hat{f_i}$ are also coprime, so for all $i,$  there exist $\alpha_i,\beta_i\in R[x]$ such that $\alpha_if_i+\beta_i\hat{f_i}=1.$ Take the case of $i=1$ for example. We have $\hat{f_1}(\varphi)(w)=\sum_{j=1}^{r}\hat{f_1}(\varphi)(w_j)=\hat{f_1}(\varphi)(w_1),$ moreover,  $(\alpha_1f_1)(\varphi)(w_1)+(\beta_1\hat{f_1})(\varphi)(w_1)=w_1,$ hence as $w_1\in V_1,$  $(\beta_1\hat{f_1})(\varphi)(w_1)=w_1,$ and thus $(\beta_1\hat{f_1})(\varphi)(w)=w_1,$ which means that $w_1 \in W\cap V_1,$ as $W$ is invariant under $\varphi$ and $w\in W.$ So by symmetry, for all $1\leq i\leq r,$ $w_i \in W\cap V_i,$ and thus $W=W\cap V_1 + \dots + W\cap V_r.$ Moreover, as the $V_i$'s are in a direct sum we get $W=W\cap V_1 \oplus \dots \oplus W\cap V_r.$

			\textit{(ii)} As both $W$ and the $V_i$'s are  $\varphi$-invariant, it is clear that the $W\cap V_i$'s are $\varphi$-invariant too.
			
			\textit{(iii)} If $W\cap V_i$ is free for all $i$ then $W$ is obviously free too. The converse follows from Proposition \ref{prop:projectiveFreeOverLocal}.
		\end{proof}

		\begin{proposition}\label{prop:associatedModuleIsFreeWhenMinIrr}
			Suppose that $\mu(\lambda_\varphi)$ is irreducible. Then $M(\varphi),$ the module associated with $\varphi,$ is free as an  $R[x]/(\lambda_\varphi)$-module.
		\end{proposition}
		\begin{proof}
			As $\mu(\lambda_\varphi)$ is irreducible, by Theorem \ref{thm:sqFreeImplyRegular}, $\varphi$ is regular and $N_\varphi=(\lambda_\varphi).$ It is clear that $(\lambda_\varphi)$ is the annihilator of $M(\varphi),$ and thus $M(\varphi)$ can naturally be viewed as an $R[x]/(\lambda_\varphi)$ module. Denote $S$ the ring $R[x]/(\lambda_\varphi).$ By Proposition \ref{prop:quotientByIrrIsSPIR}, $S$ is an SPIR with ideals $\gamma^kS=(\gamma^k + (\lambda_\varphi))$ for $0\leq k\leq \nu.$ We have then by  the structure theorem for finitely generated modules over a PIR \cite[Theorem 15.33]{brown1993matrices}, there exist an integer $q$ and $k_1\leq \dots \leq k_q $ in  $ \{1,\dots,\nu\}$ such that $M(\varphi) \simeq \bigoplus_{1\leq i\leq q} S/\gamma^{k_i}S.$  Hence, as $R$-modules, $F \simeq \bigoplus_{1\leq i\leq q} S/\gamma^{k_i}S,$ and thus, by Proposition \ref{prop:projectiveFreeOverLocal}, for all $1\leq i\leq q,$ $S/\gamma^{k_i}S$ is a free $R$-module, which implies that there exists an element $w_i\in S/\gamma^{k_i}S$ such that $\gamma^{\nu-1}w_i\neq 0,$ but as $\text{Ann}_S(S/\gamma^{k_i}S)=\gamma^{k_i}S,$ this means that for all $i,$ $k_i=\nu,$ and thus $M(\varphi)\simeq S^q$ is a free $S$-module.
		\end{proof}

		\begin{theorem}\label{thm:minSqFreeImplySemiSimple}
			The endomorphism $\varphi$ is semisimple.
		\end{theorem}
		\begin{proof}
		 Let $W$ be a free $\varphi$-invariant submodule of $F.$ From Proposition \ref{prop:invariantDirectSumWhenMinSqFree} we have that $W=\bigoplus_{i=1}^{r}W\cap V_i$ with each $W\cap V_i$ free and $\varphi$-invariant. Hence to prove that $W$ have a $\varphi$-invariant complement, it suffices to prove that for each $i,$  $W\cap V_i$ has a $\varphi$-invariant complement in $V_i.$ We can then \textit{w.l.o.g} suppose that $\lambda_\varphi$ is basic irreducible. Denote $S$ the ring $R[x]/(\lambda_\varphi).$ Since $W$ is free and $\varphi_{\restriction W}$ is regular of minimal polynomial $\lambda_\varphi,$ by Proposition \ref{prop:associatedModuleIsFreeWhenMinIrr}, $W$ can be seen as a free $S$-submodule of $M(\varphi).$ Moreover, from Proposition \ref{prop:quotientByIrrIsSPIR}, $S$ is an SPIR, hence it is weakly semi-Steinitz and by Proposition \ref{prop:steinitz}(iii), there exists an $S$-submodule $W'$ of $M(\varphi)$ such that $M(\varphi)=W\oplus W'.$ But since $W'$ is an $R[x]/(\lambda_\varphi)$-submodule of $M(\varphi),$ it can be seen as a $\varphi$-invariant $R$-module, hence as $R$-modules, $F=W\oplus W'$ with  $W'$ $\varphi$-invariant.
		\end{proof}
		
		\begin{remark}\label{rk:semiSimpleNoImplySqFree}
			While an endomorphism of a finite dimensional vector space is semisimple if and only if its minimal polynomial is square-free (see \cite[Theorem 11, Chapter 7]{hoffmann1971linear}), the converse of Theorem \ref{thm:minSqFreeImplySemiSimple} does not hold, as shown by the following counterexample. Let $R=\mathbb{Z}/4\mathbb{Z}$ and consider the matrix $A=\begin{pmatrix}
				0 & 1\\
				2 & 0
			\end{pmatrix} \in M_2(R).$ It is easy to see that $\lambda_A=x^2-2,$ and $\mu(\lambda_A)=x^2$ is not square-free. We claim that $R^2$ has no nonzero proper free $\varphi_A$-invariant submodules. Indeed otherwise, let $M$ be such a submodule. As the rank of $M$ must be $1,$ let $w=(w_1,w_2)\in R^2$ such that $M=Rw$ and $\mu(w)\neq 0.$ We have then $\varphi_A(w)=(w_2,2w_1)\in M$ which means that there exists $\alpha\in R$ such that $w_2=\alpha w_1$ and $2w_1=\alpha w_2$ and this implies that $w_1,w_2 \in 2R$ which contradicts the assumption that $\mu(w)\neq 0.$ We have then, $\varphi_A$ is trivially semisimple despite the fact that $\mu(\lambda_A)$ is not square-free.
		\end{remark}

		\section{Invariant submodules}\label{sec:invariance}

		We begin our investigation of the $\varphi$-invariant submodules of $F$ by considering the simplest type, namely, the smallest $\varphi$-invariant submodule containing a given element of $F.$				
		\begin{definition}
			 For each $w\in F,$ the \textit{$\varphi$-cyclic submodule} of $F$ generated by $w$ is the submodule of all elements of the form $g(\varphi)(w)$ with $g\in R[x].$ We denote it $Z(w,\varphi).$
		\end{definition}
		\begin{remark}
			For an element $w\in F,$  $Z(w,\varphi)$ can be viewed as an $R[x]$-submodule of $M(\varphi),$ the $R[x]$-module associated with $\varphi.$ In this setting it is easy to see that $Z(w,\varphi)$ corresponds to the cyclic module generated by $w$ in $M(\varphi),$ and conversely, for each $w\in M(\varphi),$ $R[x]w$ corresponds to $Z(w,\varphi).$
		\end{remark}
		
		For a $k$-vector space $V$ and $\psi\in \text{End}_k(V),$ for each vector $w\in V,$ we denote $N_\psi(w)$ the ideal of $k[x]$ such that $p\in N_\psi(w)$ if and only $p(\psi)(w)=0.$ The next lemma uses the fact that for each $w\in F,$ $N_{\mu(\varphi)}(w+\mathfrak{m}F)$ is a principal ideal that contains $1$ if and only if $w\in\mathfrak{m}F.$
		\begin{lemma}\label{lem:cyclicSubModWhenMinIsIrreducible}
			Suppose that $\lambda_\varphi$ is basic irreducible. Then for each nonzero element $w\in F,$  the set $\mathcal{B}=\{\varphi^i(w) ~|~0\leq i\leq \deg(\lambda_\varphi)-1\}$ 
			is a minimal generating set of $Z(w,\varphi).$ Furthermore, $Z(w,\varphi)$ is free if and only if $w \notin \mathfrak{m}F,$ and in this case, the restriction of $\varphi$ to $Z(w,\varphi)$ is a regular endomorphism with both minimal and characteristic polynomials equal to $\lambda_\varphi.$ 
		\end{lemma}
		\begin{proof}
			Suppose at first that $ w\notin  \mathfrak{m}F.$ As $\lambda_\varphi(\varphi)(w)=0$ and $\lambda_\varphi$ is monic, it is clear that $\mathcal{B}$ generates $Z(w,\varphi).$ We claim that $\mathcal{B}$ is linearly independent. Otherwise, by Corollary \ref{cor:independenceTransfer}, $\mu(\mathcal{B})=\{\varphi^i(w) + \mathfrak{m}F ~|~0\leq i\leq \deg(\lambda_\varphi)-1\}$ is linearly dependent in $F/\mathfrak{m}F.$ But from Diagram \ref{eq:diagram}, we see that $\mu(\mathcal{B})=\{\mu(\varphi^i) (w+ \mathfrak{m}F) ~|~ 0\leq i\leq \deg(\lambda_\varphi)-1\}=\{\mu(\varphi)^i (w+ \mathfrak{m}F) ~|~ 0\leq i\leq \deg(\lambda_\varphi)-1\}.$ Hence, $\mu(\mathcal{B})$ being dependent means that there exists a monic polynomial $f\in  k[x]$ such that $\deg(f) \leq \deg(\lambda_\varphi)-1$ and $f(\mu(\varphi))(w+\mathfrak{m}F)=0,$ i.e. $f\in N_{\mu(\varphi)}(w+\mathfrak{m}F).$ As $w\notin \mathfrak{m}F,$ $N_{\mu(\varphi)}(w+\mathfrak{m}F)\neq k[x],$ hence as the irreducible polynomial $\mu(\lambda_\varphi)$ is in $N_{\mu(\varphi)}(w+\mathfrak{m}F)$ we have $ N_{\mu(\varphi)}(w+\mathfrak{m}F)=(\mu(\lambda_\varphi)),$ and that implies that $f \in (\mu(\lambda_\varphi))$ which contradicts the fact that $\deg(f) \leq \deg(\lambda_\varphi)-1.$ We have then $Z(w,\varphi)$ is a free submodule of $F$ with basis $\mathcal{B}.$ Moreover, $\varphi_{\restriction Z(w,\varphi)}$ is annihilated by $\lambda_\varphi,$ hence by Theorem \ref{thm:sqFreeImplyRegular}, it is regular, and as $\operatorname{rank}_R(Z(w,\varphi))=\deg(\lambda_\varphi)$ and $\lambda_\varphi$ is irreducible, $\lambda_{\varphi_{\restriction Z(w,\varphi)}}=\rchi_{\varphi_{\restriction Z(w,\varphi)}}=\lambda_\varphi.$    
						
				Now suppose that $w\in \mathfrak{m}F.$ By Lemma \ref{lem:uniqueExpressionSPIR}, there exist $0\leq s\leq \nu-1$ and $w'\in F\setminus \mathfrak{m}F$ such that $w=\gamma^sw'.$ It is straightforward to see that $Z(w,\varphi)=Z(\gamma^sw',\varphi)=\mathfrak{m}^sZ(w',\varphi).$ So as $Z(w',\varphi)$ is free with basis $\{\varphi^i(w') ~|~0\leq i\leq \deg(\lambda_\varphi)-1\},$ multiplying by $\gamma^s$ yields by Proposition \ref{prop:moduleChainHasSameMinGenNum} that $\{\varphi^i(w) ~|~0\leq i\leq \deg(\lambda_\varphi)-1\}$ is a minimal generating set of $Z(w,\varphi).$ Moreover, as this minimal generating set is clearly linearly dependent, by Lemma \ref{lem:minimalGenSetIsLinIndep}, $Z(w,\varphi)$ is not free. 
		\end{proof}
		
		\begin{lemma}\label{lem:sumCoprimeCyclicSubModules}
			 For all $1\leq i\leq r,$ let $w_i \in V_i,$ and denote $w=w_1+\dots+w_r.$ We have $Z(w,\varphi)=Z(w_1,\varphi)\oplus\dots\oplus Z(w_r,\varphi).$ Moreover, $Z(w,\varphi)$ is free if and only if, for every $i,$ either $w_i=0$ or $w_i\notin \mathfrak{m}V_i.$
		\end{lemma}
		
		\begin{proof}
			It is clear that $Z(w,\varphi)\subseteq Z(w_1,\varphi)+\dots+ Z(w_r,\varphi).$ For $i\in \{1,\dots,r\},$ as $f_1,\dots,f_r$ are pairwise coprime, there exists $u,v\in R[x]$ such that  $uf_i + v\prod_{\substack{1\leq j\leq r \\ j\neq i}}f_j = 1.$ As $Z(w,\varphi)$ is $\varphi$-invariant, $(v\prod_{\substack{1\leq j\leq r \\ j\neq i}}f_j)(\varphi)(w)=(v\prod_{\substack{1\leq j\leq r \\ j\neq i}}f_j)(\varphi)(w_i)=w_i-(uf_i)(\varphi)(w_i)=w_i\in Z(w,\varphi).$ Hence, $Z(w,\varphi)= Z(w_1,\varphi)+\dots+ Z(w_r,\varphi),$ and as $Z(w_i,\varphi)\subseteq V_i$ for all $i,$ by Proposition \ref{prop:kernelsProp}, the sum is direct. Moreover, as $V_i$ is free, using Lemma $\ref{lem:cyclicSubModWhenMinIsIrreducible}$, if $w_i\notin \mathfrak{m}V_i$ for all $i,$ then $Z(w,\varphi)$ is free. Conversely, by Proposition \ref{prop:projectiveFreeOverLocal}, if $Z(w,\varphi)$ is free, then $Z(w_i,\varphi)$ is also free for all $i,$ hence once again by Lemma \ref{lem:cyclicSubModWhenMinIsIrreducible},  $w_i\notin \mathfrak{m}V_i$ for all $i.$
		\end{proof}

		\begin{proposition}\label{prop:cyclicSubModWhenMinIsSqFree}
		 We have the following:
			\begin{itemize}
				\item[(i)] For each nonzero element $w\in F,$ let $w=w_1+\dots+w_r$ be such that $w_i\in V_i$ for all $i,$ and denote $J=\{j :  w_j \neq 0\}.$ Then  $\{\varphi^i(w)~|~0\leq i < \sum_{j\in J}\deg(f_j)\}$ is a minimal generating set of $Z(w,\varphi).$ Moreover, $Z(w,\varphi)$ is free if and only if for all $j\in J,$ $w_j\notin \mathfrak{m}V_j,$ and in this case, the restriction of $\varphi$ to $Z(w,\varphi)$ is regular  with both minimal and characteristic polynomials equal to $\prod_{j\in J}f_j.$
				\item[(ii)] Conversely, for every nonempty subset $J$ of $\{1,\dots,r\},$ there exists an element $w\in F$ such that $Z(w,\varphi)$ is free and the restriction of $\varphi$  to $Z(w,\varphi)$ is a regular endomorphism of minimal polynomial $\prod_{j\in J}f_j.$
				
			\end{itemize}

		\end{proposition}
		\begin{proof}
			\textit{(i)} From Lemma \ref{lem:sumCoprimeCyclicSubModules}, $Z(w,\varphi)=\bigoplus_{j\in J} Z(w_j,\varphi)$ and it is free if and only if $w_j\notin \mathfrak{m}V_j$ for all $j\in J.$ By Proposition \ref{prop:kernelsProp}, $V_j$ is free and for all $j,$ $\varphi_j$  is regular with minimal polynomial $f_j.$ Hence by Lemma \ref{lem:cyclicSubModWhenMinIsIrreducible}, for all $j\in J,$ $\{\varphi^i(w_j) ~|~0\leq i< \deg(f_j)\}$ is a minimal generating set of $Z(w_j,\varphi),$ and thus by Theorem \ref{thm:matsumura2.3},  $Z(w,\varphi)$ has minimal number of generators $\sum_{j\in J}\deg(f_j).$ Moreover, as $\prod_{j\in J}f_j$ annihilates $Z(w,\varphi),$ $\{\varphi^i(w)~|~0\leq i < \sum_{j\in J}\deg(f_j)\}$ generates $Z(w,\varphi),$ hence as it has cardinality $\sum_{j\in J}\deg(f_j),$ it is a minimal generating set. The rest follows exactly like Lemma \ref{lem:cyclicSubModWhenMinIsIrreducible}.
			
			\textit{(ii)} For every $j\in J,$ let $w_j \in V_j \setminus \mathfrak{m}V_j,$ and let $w=\sum_{j\in J}w_j.$ Using \textit{(i)}, we can easily check that $Z(w,\varphi)$ verifies the statement.	
		\end{proof}
		
		The following corollary concerning the minimal number of generators of $\varphi$-cyclic modules is immediate.
		\begin{corollary}\label{cor:minNumOfGenOfCyclicSubModWhenMinIsSqFree}
			\begin{equation*}
				\left\{\nu_R(Z(w,\varphi)) ~|~ w\in F\right\} = \left\{\sum_{j\in J}\deg f_j ~|~ J\subseteq \{1,\dots,r\} \right\}.
			\end{equation*} 
		\end{corollary}

		Having examined the $\varphi$-cyclic submodules of $F,$ our next goal is to leverage them to characterize general $\varphi$-invariant submodules. A key step toward this is provided by the following proposition.

			\begin{proposition}\label{prop:decompAsDirectSumOfPhiCylicSubsWhenMinIrr}
				Suppose that $\lambda_\varphi$ is basic irreducible, and let $m\in \mathbb{N}$ such that $\rchi_{\mu(\varphi)}=\mu(\lambda_\varphi)^m.$ Then there exist $w_1,\dots,w_m \in F\setminus \mathfrak{m}F$ such that $F=Z(w_1,\varphi)\oplus \dots\oplus Z(w_m,\varphi).$
				In particular, as an $R[x]/(\lambda_\varphi)$-module, $M(\varphi)$ has basis $\{w_1,\dots,w_m\}$ and is of rank $m.$
			\end{proposition}
			\begin{proof}
				Let $w_1 \in F \setminus \mathfrak{m}F.$ By Lemma \ref{lem:cyclicSubModWhenMinIsIrreducible}, $Z(w_1,\varphi)$ is free and of rank $\deg(\lambda_\varphi).$ Hence by Theorem \ref{thm:minSqFreeImplySemiSimple}, there exists a free $\varphi$-invariant submodule $F_1 \le F$ of rank $(m-1)\deg(\lambda_\varphi)$ such that $F=Z(w_1,\varphi) \oplus F_1.$ Furthermore, $\varphi_{\restriction F_1}$ is a regular endomorphism of $F_1$ with minimal polynomial $\lambda_\varphi.$ Hence repeating the same argument yields that there exist $w_2,\dots,w_m \in F \setminus \mathfrak{m}F$ such that $F=\bigoplus_{i=1}^m Z(w_i,\varphi).$ Moreover, from Proposition \ref{prop:associatedModuleIsFreeWhenMinIrr}, $M(\varphi)$ is a free $(R[x]/(\lambda_\varphi))$-module, and it is easy to see that $M(\varphi)=\bigoplus_{i=1}^m(R[x]/(\lambda_\varphi))w_i.$ Hence $\{w_1,\dots,w_m\}$ is a basis of $M(\varphi)$ and $\operatorname{rank}_{R[x]/(\lambda_\varphi)}(M(\varphi))=m.$
			\end{proof}

			\begin{theorem}\label{thm:invSubIsDirSumOfCycSub}
			For all $1\leq i\leq r,$ let $w_{i,1},\dots,w_{i,m_i} \in V_i,$ such that $V_i=\bigoplus_{1\leq j\leq m_i}Z(w_{i,j},\varphi).$ We have for each $\varphi$-invariant submodule $W$ of $F,$ $W$ is a direct sum of submodules $W_1,\dots,W_r$ such that for all $1\leq i\leq r,$ there exist an automorphism $\psi_i$ of $V_i$ commuting with $\varphi_i$  and integers $k_1,\dots,k_{m_i} \in \{0,\dots,\nu\}$ such that 
			\begin{equation}
				W_i=\bigoplus_{1\leq j\leq m_i}\gamma^{k_j}Z(\psi_i(w_{i,j}),\varphi)=\bigoplus_{1\leq j\leq m_i}Z(\psi_i(\gamma^{k_j}w_{i,j}),\varphi).
			\end{equation}
		Conversely, every submodule of this form is $\varphi$-invariant. 
		\end{theorem}
		\begin{proof}
			Following Proposition \ref{prop:invariantDirectSumWhenMinSqFree}, let $W=\bigoplus_{1\leq i\leq r} W_i$
			 with $W_i=W \cap V_i$ for all $i.$ Fix $i\in \{1,\dots,r\}$ and denote $R_i$ the SPIR $R[x]/(f_i).$ As $W_i$ is a $\varphi_i$-invariant submodule of $V_i,$ it can be seen as an $R_i$-submodule of $M(\varphi_i).$ On the other hand, by Proposition \ref{prop:decompAsDirectSumOfPhiCylicSubsWhenMinIrr}, $\{w_{i,1},\dots,w_{i,m_i}\}$ is an $R_i$-basis of $M(\varphi_i).$ By Lemma \ref{lem:numMinGenSubModSPIR}, denoting $s_i=\mu_{R_i}(W_i),$ we have $s_i\leq m_i,$ hence by Theorem \ref{thm:subModEnumFromAGivenBasis},  there exist an $R_i$-automorphism $\psi_i$ of $V_i$  and integers $k_1,\dots,k_{s_i} \in \{0,\dots,\nu-1\}$ such that $\{\gamma^{k_1}\psi_i(w_{i,1}),\dots, \gamma^{k_{s_i}}\psi_i(w_{i,s_i})\}$ is a minimal generating set of $W_i$ when seen as an $R_i$-module. Hence as $\{\psi_i(w_{i,1}),\dots,\psi_i(w_{i,s_i})\}$ is $R_i$-linearly independent, the $R_i$-module $W_i$ is a direct sum $ R_i\gamma^{k_1}\psi_i(w_{i,1})\oplus\dots\oplus R_i\gamma^{k_{s_i}}\psi_i(w_{i,s_i}) ,$ and thus, as $R$-modules,  $W_i=\bigoplus_{1\leq j\leq s_i}Z(\psi_i(\gamma^{k_j}w_{i,j}),\varphi)=\bigoplus_{1\leq j\leq m_i}Z(\psi_i(\gamma^{k_j}w_{i,j}),\varphi)$ with $k_{s_i+1}=\dots=k_{m_i}=\nu.$ Moreover, as $\psi_i$ is an $R_i$-automorphism of $V_i,$ it is an $R$-automorphism of $V_i$ that commutes with $\varphi_i.$ The converse is immediate.		 
		\end{proof}

		As an immediate consequence, we obtain the minimal number of generators of each $\varphi$-invariant submodule of $F.$
		
		\begin{corollary}\label{cor:minNumOfGenOfInvariantSubModWhenMinIsSqFree}
		Denote $\mathcal{I}(\varphi)$ the set of $\varphi$-invariant submodules of $F.$ We have: 
		\begin{equation}\label{eq:numMinGenPhiInvSubMod}
			\left\{\nu_R(W) ~|~ W\in \mathcal{I}(\varphi)\right\} = \left\{\sum_{i=1}^{r}s_i\deg(f_i) ~|~ 0\leq s_i\leq m_i, \text{ for all } i \right\}.
		\end{equation}
		\end{corollary}

		When the ring $R$ is finite, in addition to the characterization of Theorem \ref{thm:invSubIsDirSumOfCycSub},  we can give the number of the free $\varphi$-invariant submodules of $F.$ Recall that the Gaussian $q$-binomial coefficient for $0\leq m\leq n$ is defined as $\displaystyle \binom{n}{m}_{q}=\prod_{i=0}^{m-1} \frac{q^n - q^i}{q^m - q^i}.$
			\begin{proposition}
			Suppose that $R$ is a finite chain ring with residue field $\mathbb{F}_q,$ and for all $1\leq i\leq r,$ let $d_i$ denote $\deg(f_i).$ Then denoting $\mathcal{F}(\varphi)$ the set of free $\varphi$-invariant submodules of $F,$ we have:   
			\begin{equation}
				\lvert \mathcal{F}(\varphi)\rvert = \prod_{i=1}^{r}\left(\sum_{j=0}^{m_i}q^{d_i(\nu-1)j(m_i-j)}\binom{m_i}{j}_{q^{d_i}}\right).
			\end{equation}
		\end{proposition}
		\begin{proof}For all $1\leq i\leq r,$ let $S_i$  denote the ring $R[x]/(f_i)$ and $\mathcal{F}_i$ the set of free $\varphi_i$-invariant submodules of $V_i.$	
		From propositions \ref{prop:projectiveFreeOverLocal} and  \ref{prop:invariantDirectSumWhenMinSqFree}, we see that $\mathcal{F}(\varphi)$  has a one-to-one correspondence with $\mathcal{F}_1\times \mathcal{F}_2\times\dots\times \mathcal{F}_r,$ hence $\lvert \mathcal{F}(\varphi) \rvert= \prod_{i=1}^{r} \lvert \mathcal{F}_i \rvert.$ 
		
		Moreover for all $i,$ a submodule of $V_i$ is free and $\varphi_i$-invariant if and only if it is a free $S_i$-submodule of $M(\varphi_i).$ Indeed, by Theorem \ref{thm:minSqFreeImplySemiSimple}, $\varphi_i$ is semisimple, so for each free $\varphi_i$-invariant submodule $W$ of $V_i,$ there exists a $\varphi_i$-invariant submodule $W'$ such that $V_i=W\oplus W',$ hence as $S_i$-modules, we have $M(\varphi_i)=W\oplus W',$ but by Proposition \ref{prop:associatedModuleIsFreeWhenMinIrr}, $M(\varphi_i)$ is free, hence $W$ is a free $S_i$-module. Conversely, as $S_i$ is an SPIR, by Proposition \ref{prop:steinitz} a free submodule of $M(\varphi_i)$ has a complement, hence it is a $\varphi_i$-invariant direct summand of $V_i$ which implies that it is a free $R$-module. We have then, the cardinality of $\mathcal{F}_i$ equals the number of the free $S_i$-submodules of $M(\varphi_i).$  On the other hand, the residue field of $S_i$ is isomorphic to $k[x]/(\mu(f_i))\simeq\mathbb{F}_{q^{d_i}},$ hence as $\operatorname{rank}_{S_i}(M(\varphi_i))=m_i$ by Proposition \ref{prop:decompAsDirectSumOfPhiCylicSubsWhenMinIrr}, we have by \cite[Proposition 3]{byrne2022density}, $\lvert \mathcal{F}_i \rvert= \sum_{j=0}^{m_i}q^{d_i(\nu-1)j(m_i-j)}\binom{m_i}{j}_{q^{d_i}},$ and the cardinality of $\mathcal{F}(\varphi)$ follows.
		\end{proof}

		Another consequence of the decomposition of $F$ as a direct sum of $\varphi$-cyclic submodules provided in Proposition \ref{prop:decompAsDirectSumOfPhiCylicSubsWhenMinIrr} is that we can sharpen Proposition \ref{prop:squareFreeAnnihilatorImpliesNormal} as follows.

		\begin{theorem}\label{thm:minIrrImpliesPhiSimToMatDiagCompanionBlocks}
			 A basis of $F$ may be chosen relative to which $\varphi$ is represented by the matrix $\operatorname{diag}(C_1,\dots,C_r)$ such that for all $1\leq i\leq r,$ 	
			
			 \begin{equation}
			 	C_i=
			 \left.
			 \begin{bmatrix}
			 	\operatorname{com}(f_i) & &  \text{\large 0} \\ & \ddots & \\ \text{\large 0} & & \operatorname{com}(f_i) 
			 \end{bmatrix}
			 \right\} m_i \text{ times}.
			 	\end{equation}

		\end{theorem}
		
		\begin{proof}
			By Proposition \ref{prop:kernelsProp}, there exists a basis of $F$ in which $\varphi$ is represented by a matrix $\operatorname{diag}(C'_1,\dots,C'_r)$ such that for all $1\leq i\leq r,$ $C'_i$ is the matrix of $\varphi_i$ in some basis of $V_i$. Fix $1\leq i\leq r.$ As $\rchi_{\mu(\varphi_i)}=\mu(f_i)^{m_i},$ by Proposition \ref{prop:decompAsDirectSumOfPhiCylicSubsWhenMinIrr}, let $w_1,\dots,w_{m_i}\in V_i$ such that $V_i=\bigoplus_{j=1}^{m_i}Z(w_j,\varphi).$ Then it follows immediately from Lemma \ref{lem:cyclicSubModWhenMinIsIrreducible} and Proposition \ref{prop:charaSqFreeImpliesRegular} that $\varphi_i$ can be represented by the matrix $\operatorname{diag}(\underbrace{\operatorname{com}(f_i), \dots, \operatorname{com}(f_i)}_{m_i \text{ times}})$ in some basis of $V_i.$		
 		\end{proof}

 		\begin{corollary}
 			Suppose that the residue field $k$ of $R$ is algebraically closed. Then every matrix $A\in M_n(R)$ with a residually square-free annihilating polynomial is diagonalizable.
 		\end{corollary}
 		\begin{proof}
 			Follows immediately from Theorem \ref{thm:minIrrImpliesPhiSimToMatDiagCompanionBlocks} and the fact that when $k$ is algebraically closed, every monic basic irreducible polynomial in $R[x]$ has degree $1.$
 		\end{proof}

 		\begin{corollary}\label{cor:charaPolOfSqFreeMin}
 			The characteristic polynomial of $\varphi$ is  $\rchi_\varphi=\prod_{i=1}^{r}f_i^{m_i}.$ 
 		\end{corollary}
 		\begin{proof}
 			Follows from Theorem \ref{thm:minIrrImpliesPhiSimToMatDiagCompanionBlocks}, Proposition \ref{prop:charaDividesChara} and the fact that $\rchi_{\operatorname{com}(f_i)}=f_i$ for all $i.$
 		\end{proof}
 		
 		\begin{remark}
 			An endomorphism with no residually square-free annihilating polynomial may have a characteristic polynomial that cannot be expressed as a product of basic irreducible polynomials. Indeed, by use of the companion matrix, each monic polynomial of degree equal to the rank of $F$ is the characteristic polynomial for some endomorphism of $F.$ And in the ring $R=\mathbb{Z}/4\mathbb{Z}$ for example, $x^2+2$ cannot be written as a product of basic irreducible polynomials. 		\end{remark}

 		We end this section by considering the question of similarity in the context of matrices with residually square-free annihilating polynomials. The next theorem generalizes the results of \cite{davis1968certain} from $\mathbb{Z}/p^n\mathbb{Z}$ to SPIRs.
 		
 		\begin{theorem}\label{thm:similarityWhenMinPolSqFree}
 			Let $A,B \in M_n(R)$ be two matrices with residually square-free annihilating polynomials. The following are equivalent:
 			\begin{itemize}
 				\item[(i)] $A$ and $B$ are similar,
 				\item[(ii)] $\lambda_A=\lambda_B$  and $\mu(A)$ and $\mu(B)$ are similar,
 				\item[(iii)] $\lambda_A=\lambda_B$  and $\rchi_A=\rchi_B,$ 	
 				\item[(iv)] $\lambda_A=\lambda_B$  and $\mu(\rchi_A)=\mu(\rchi_B).$ 		
 			\end{itemize}
 		\end{theorem}
 		\begin{proof}
 			\textit{(i)} clearly implies \textit{(ii)} and 	\textit{(iii)} clearly implies \textit{(iv)}.
 			
			$(ii) \implies (iii)$ As $\mu(A)\sim \mu(B),$ $\rchi_{\mu(A)}=\rchi_{\mu(B)}.$ Hence from Corollary \ref{cor:charaPolOfSqFreeMin}, as $\lambda_A=\lambda_B,$ we have that $\rchi_A=\rchi_B.$
			
			$(iv) \implies (i)$ Follows immediately from Theorem \ref{thm:minIrrImpliesPhiSimToMatDiagCompanionBlocks}.
 		\end{proof}

 		\begin{remark}
 			While over a field, two linear operators with square-free minimal polynomials are similar if and only if their characteristic polynomials are equal, this does not hold for $R$ in general. This stems from the fact that $R[x]$ is not always a UFD and the characteristic polynomial can have multiple factorizations into basic irreducible polynomials. Consider for example $R=\mathbb{Z}/4\mathbb{Z},$ and let $A=\begin{pmatrix}
 				3 & 0\\
 				0 & 3
 			\end{pmatrix}.$ We have that  both $A$ and $I_2$ have residually square-free annihilating polynomials and  $\rchi_A=(x-3)^2=(x-1)^2=\rchi_{I_2},$ yet $\lambda_A=x-3\neq \lambda_{I_2}=x-1,$ which means that $A$ and $I_2$ cannot be similar.
 		\end{remark}

		\section{Invariance in the non-derogatory case}\label{sec:nonDerogInvariance}

		 A regular endomorphism such that the minimal and characteristic polynomials coincide is said to be \textit{non-derogatory}. We devote this section to investigating this special case.
	
		\begin{lemma}\label{lem:cyclicSubModNonDerogMinIrred}
			Suppose that $\varphi$ is non-derogatory and that $\lambda_\varphi$ is basic irreducible. We have
			\begin{equation}
				\{Z(w,\varphi) ~|~ w\in F\}=\{\mathfrak{m}^i F ~|~ 0\leq i\leq \nu\}.
			\end{equation} 
			In particular, for $w\in F,$ if $i$ is an integer such that $w \in \mathfrak{m}^i F \setminus \mathfrak{m}^{i+1}F$ then $Z(w,\varphi)=\mathfrak{m}^i F.$
		\end{lemma}
		\begin{proof}
			Let $w\in F.$ Following Lemma \ref{lem:uniqueExpressionSPIR}, let $0\leq i\leq \nu$ and $w'\in F$ such that $w=\gamma^iw'$ and $\mu(w')\neq 0.$ By Lemma \ref{lem:cyclicSubModWhenMinIsIrreducible}, we have that $Z(w',\varphi)$ is free and of rank $\deg(\lambda_\varphi).$ But as $\varphi$ is non-derogatory $\deg(\lambda_\varphi)=\deg(\rchi_\varphi)=rank_R(F).$ Hence $F=Z(w',\varphi)$ which implies that $\mathfrak{m}^i F= \mathfrak{m}^iZ(w',\varphi)=Z(\gamma^iw',\varphi)=Z(w,\varphi).$ Conversely, it is easy to see by Lemma \ref{lem:cyclicSubModWhenMinIsIrreducible}, that for all $0\leq i\leq \nu,$ for each $w\in F \setminus \mathfrak{m}F,$ $Z(\gamma^iw,\varphi)=\mathfrak{m}^iF.$
		\end{proof}
		
		\begin{theorem}\label{thm:invariantSubmodulesNonDerogMinSqFree}
			Suppose that $\varphi$ is non-derogatory, and denote $\mathcal{I}(\varphi)$ the set of $\varphi$-invariant submodules of $F.$ We have then $\mathcal{I}(\varphi) = \left\{ Z(w,\varphi) ~|~ w\in F \right\}.$ Moreover, every $\varphi$-invariant submodule $W$ of $F$ is of the form
			\begin{equation}\label{eq:invariantSubmodulesNonDerogMinSqFree}
				W = \bigoplus_{i=1}^{r}\mathfrak{m}^{k_i}ker(f_i(\varphi)) ~~~~~~~~ \text{with } 0\leq k_i \leq \nu, \text{ for all } i.
			\end{equation}
		\end{theorem}
		\begin{proof}
		Obviously, each $\varphi$-cyclic submodule of $F$ is $\varphi$-invariant.		
		Conversely, let $W\le F$ be a $\varphi$-invariant submodule. As $\varphi$ is non-derogatory, from Theorem \ref{thm:invSubIsDirSumOfCycSub}, for all $1\leq i\leq r,$ there exists an element $w_i\in V_i$ such that $W=\bigoplus_{i=1}^rZ(w_i,\varphi),$ which implies by Lemma \ref{lem:sumCoprimeCyclicSubModules} that $W=Z(w_1+\dots+w_r,\varphi).$

		 Equation (\ref{eq:invariantSubmodulesNonDerogMinSqFree}) follows then immediately from lemmas \ref{lem:sumCoprimeCyclicSubModules} and \ref{lem:cyclicSubModNonDerogMinIrred}.
		\end{proof}
		
		\begin{remark}
			A polycyclic code of length $n$ associated to a polynomial $f$ over a finite chain ring $R$ can be seen as a submodule of $R^n$ that is invariant under the action of the matrix $\operatorname{com}(f).$ When $f$ is residually square-free, the previous theorem corrects a result given in \cite[Theorem 4.2]{EdgarMaryamSteve} concerning the structure of the polycyclic codes associated to $f,$ where each $k_i$ in Equation (\ref{eq:invariantSubmodulesNonDerogMinSqFree}) was taken to be either $\nu$ or $0.$
		\end{remark}
		
		The next corollary follows immediately.
		\begin{corollary}
			Suppose that $\varphi$ is non-derogatory. The following hold.
			\begin{itemize}
				\item[(a)] $\{V_i ~|~ 1\leq i\leq r\}$ is the set of all minimal nonzero free $\varphi$-invariant submodule of $F.$
				\item[(b)] Every free $\varphi$-invariant submodule of $F$ is of the form $\bigoplus_{j\in J}V_j$ with $J\subseteq \{1,\dots,r\}.$
				\item[(c)] There are $(\nu+1)^r$ $\varphi$-invariant submodules of $F,$ among which $2^r$ are free.
			\end{itemize}
		\end{corollary}

		Recall that a submodule $W$ of $F$ is called $\varphi$-hyperinvariant if it is invariant under every $R$-module endomorphism $\psi \in \text{End}_R(F)$ that commutes with $\varphi.$
		\begin{proposition}
			Suppose that $\varphi$ is non-derogatory. Then every $\varphi$-invariant submodule $W$ of $F$ is $\varphi$-hyperinvariant. 	
		\end{proposition}
		\begin{proof}
			Let $W$ be a $\varphi$-invariant submodule of $F$ and let $\psi$ be an endomorphism of $F$ that commutes with $\varphi.$ Following Theorem \ref{thm:invariantSubmodulesNonDerogMinSqFree}, let $k_1,\dots,k_r\in \mathbb{N}$ such that $W= \bigoplus_{i=1}^{r}\mathfrak{m}^{k_i}V_i.$ We have then that for every element $w\in W,$ for each $1\leq i\leq r,$ there exists $w_i\in V_i$ such that $w=\bigoplus_{i=1}^{r}\gamma^{k_i}w_i.$ On the other hand, $\psi$ commutes with every polynomial in $\varphi,$ hence we have for all $i,$ $f_i(\varphi)(\psi(w_i))=\psi(f_i(\varphi)(w_i))=0,$ which implies that $\psi(w_i)\in V_i.$ We have then, $\psi(w)=\psi\left(\sum_{i=1}^{r}\gamma^{k_i}w_i\right)=\sum_{i=1}^{r}\gamma^{k_i}\psi(w_i)\in \bigoplus_{i=1}^{r}\mathfrak{m}^{k_i}V_i=W,$ and thus $W$ is $\psi$-invariant.
		\end{proof}

	\section*{Acknowledgments}
		The first author thanks Hassan Ou-azzou for his many helpful remarks and suggestions that helped improve this paper.

	\printbibliography

	\end{document}